\documentclass[psamsfonts,12pt]{amsart}

\usepackage[margin=1in]{geometry}  
\usepackage{graphicx}              
\usepackage{amsmath}               
\usepackage{amsfonts}              
\usepackage{amsthm}                
\usepackage{verbatim}              
\usepackage{indentfirst}           
\usepackage{underscore}
\usepackage{enumitem}
\usepackage{mathtools} 
\usepackage{amssymb}
\usepackage[all]{xy}
\usepackage{hyperref}
\usepackage{caption}
\usepackage{setspace}
\usepackage{mathrsfs}

\newtheorem{thm}{Theorem}[section]
\newtheorem{lem}[thm]{Lemma}
\newtheorem{prop}[thm]{Proposition}
\newtheorem*{definition*}{Definition}
\newtheorem{cor}[thm]{Corollary}
\newtheorem{conj}[thm]{Conjecture}

\theoremstyle{remark}
       
\theoremstyle{remark}

\newcommand{\llra}{\Longleftrightarrow}

\newcommand{\mP}{\mathbb{P}}

\newcommand{\Q}{\mathbb{Q}}
\newcommand{\C}{\mathbb{C}}
\newcommand{\fr}{\frac}

\newcommand{\CC}{\mathbb{C}}

\usepackage{amssymb,fge}

\title{A lower bound on the canonical height for polynomials}

\author[Nicole R. Looper]{Nicole R. Looper}
\address{Department of Mathematics, Northwestern University; 2033 Sheridan Road, Evanston, IL 60208, USA}
\email{nlooper@math.northwestern.edu}
	
\begin{document}

		\begin{abstract} \normalsize We prove a lower bound on the canonical height associated to polynomials over number fields evaluated at points with infinite forward orbit. The lower bound depends only on the degree of the polynomial, the degree of the number field, and the number of places of bad reduction.
			
		\end{abstract}
		\maketitle
		\renewcommand{\thefootnote}{}
		\footnote{Research partially supported by an NSF Graduate Research Fellowship.}

		\section{Introduction}
		
		The canonical height associated to a rational morphism $\phi:\mP^1\to\mP^1$ of degree at least two over a number field $K$ is a function $\hat{h}_\phi:\mP^1(\overline{K})\to\mathbb{R}$ satisfying \[\hat{h}_\phi(\phi(\alpha))=\textup{deg}(\phi)\hat{h}_\phi(\alpha) \textup{ \hspace{3mm} and \hspace{3mm} } \hat{h}_\phi(\alpha)=h(\alpha)+O(1),\]  where $h$ is the Weil logarithmic height \cite{CallSilverman}. One basic property of $\hat{h}_\phi$ is that $\hat{h}_\phi(\alpha)=0$ if and only if $\phi^i(\alpha)=\phi^j(\alpha)$ for some $i\ne j$. A natural question then arises: how small can $\hat{h}_\phi(\alpha)$ be under the assumption that $\alpha\in K$ has infinite forward orbit under $\phi$? We will address how the minimal possible nonzero value of $\hat{h}_\phi(\alpha)$ grows with the height of $\phi$ in the moduli space $\mathcal{M}_d$ of degree $d$ dynamical systems.
		
		The canonical height attached to $\phi$ has an analogue in the setting of elliptic curves. Lang conjectured that for an elliptic curve $E$ over a number field $K$ with minimal discriminant $\mathscr{D}_{E/K}$, \[\hat{h}(P)\ge c_1\log|\textup{Norm}_{K/\Q}\mathscr{D}_{E/K}|+c_2\] for any non-torsion point $P\in E(K)$, where $c_1>0$ and $c_2$ depend only on $K$ \cite[p. 92]{Lang}. In \cite{Silverman2}, Silverman gives a partial solution to this conjecture, showing that it holds for constants $c_1>0$ and $c_2$ depending on $[K:\Q]$ and on the number of primes at which $E/K$ has split multiplicative reduction. 
		
		In the dynamical setting, Silverman has made the following conjecture \cite[\S4.11]{Silverman3}. 
		
		\begin{conj}{\label{conj:mincanht}} Let $h_{\mathcal{M}_d}$ be the height function associated to an embedding of the space of degree $d\ge 2$ rational maps $\mathcal{M}_d$ into projective space, and let $K$ be a number field. There exists a positive constant $c$ depending only on $K$, $d$, and the choice of $h_{\mathcal{M}_d}$, such that for all rational maps $\phi\in K(z)$ of degree $d$, \[\hat{h}_\phi(P)\ge c\max\{h_{\mathcal{M}_d}(\langle \phi\rangle),1\}\] for any non-preperiodic point $P\in\mP^1(K)$.
		\end{conj}

In this article, we prove a weaker version of this claim in the case of polynomials, analogously to the aforementioned result of Silverman.
		
		\begin{thm}{\label{thm:minht}} Let $h_{\mathcal{M}_d}$ be the height function associated to an embedding of the space of degree $d\ge 2$ rational maps $\mathcal{M}_d$ into projective space, and let $K$ be a number field. Suppose $\phi\in K[z]$ is a polynomial of degree $d$ and has $s$ or fewer places of bad reduction. There exist constants $\kappa_1>0$ and $\kappa_2$ depending only on $d$, $s$, $[K:\Q]$, and the choice of $h_{\mathcal{M}_d}$ such that \begin{equation}{\label{eqn:mainhtbound}}\hat{h}_\phi(\alpha)\ge\kappa_1h_{\mathcal{M}_d}(\langle\phi\rangle)+\kappa_2\end{equation} for any $\alpha\in K$ having infinite forward orbit under $\phi$.
			
		\end{thm}
	In particular, we recover full uniformity across rational maps conjugate to a polynomial having everywhere good reduction. We remark that in \cite{HindrySilverman}, Hindry and Silverman prove that the full statement of Lang's Conjecture follows from the $abc$-Conjecture. One might wonder whether a similar assumption would give a proof of Conjecture \ref{conj:mincanht} for polynomials.
	
	The strategy behind the proof of Theorem \ref{thm:minht} draws its inspiration from the central ideas of \cite{Ingram2}, which concerns the map $\phi(z)=z^d+c$. Ingram shows that if many iterates $\phi^i(\alpha)$ have small local height at some given valuation $v$, then any pre-image of these $\phi^i(\alpha)$ under $\phi$ must be $v$-adically very close to some root of $\phi$. Moreover, the degree of closeness increases suitably as the local height of $c$ grows. It follows from the pigeon-hole principle that some explicit proportion of the $\phi^i(\alpha)$ must be very close together in the $v$-adic metric, in a way reminiscent of Lemmas 3 and 4 of \cite{Silverman2}. The product formula then yields a lower bound on the global canonical height.
	
	For degree $d$ polynomials not of the form $z^d+c$, this approach fails: one cannot conclude that all pre-images of low local height iterates $\phi^i(\alpha)$ lie close to a root of $\phi$. Instead, working with representatives of $\langle\phi\rangle\in\mathcal{M}_d$ in a certain normal form $f_\textbf{c}$, any pre-image under $f_\textbf{c}^3$ of a low local height point $f_\textbf{c}^i(\alpha)$ necessarily lies near some root of $f_\textbf{c}^3$. The key to this argument is an analysis of the equipotential curves of the local Green's functions of $f_\textbf{c}$. Assuming at least some critical point of $f_\textbf{c}$ has $v$-adically unbounded orbit, the level sets corresponding to points in the grand orbit of the fastest escaping critical point form the boundaries of adjacent annuli. These sequences of annuli are endowed with a natural a tree structure, which was analyzed in depth in \cite{BrannerHubbard}, \cite{DeMarco}, and \cite{DM}. By tracking how the moduli of the annuli grow as the local height of $f_\textbf{c}$ grows, we argue that any $f_\textbf{c}^i(\alpha)$ with low local height at $v$ must have some pre-image under $f_\textbf{c}^3$ that is $v$-adically close to a root of $f_\textbf{c}^3$. We are able to implement this idea in both the archimedean and non-archimedean settings, allowing us to incorporate non-archimedean places of bad reduction in the statement of Theorem \ref{thm:minht}.
	\newline
	
	\indent\textbf{Acknowledgements.} I would like to thank Laura DeMarco for introducing me to the trees in \cite{DeMarco} and \cite{DM}, and for  many helpful and enlightening discussions on that topic. I would also like to thank Patrick Ingram for several fruitful conversations on this research problem, and for sharing his results concerning the normal form used throughout this article. Finally, it is a pleasure to thank Laura DeMarco and Joseph Silverman for their useful comments on an earlier draft.

\section{Background}

Let $\textup{Rat}_d$ be the set of degree $d$ rational maps $\phi:\mP^1(\C)\to\mP^1(\C)$, where $d\ge 2$. Under its natural embedding into $\mathbb{P}^{2d+1}(\C)$, it is the complement of the hypersurface given by the resultant locus, and hence has the structure of an affine variety. The action of conjugation by M\"{o}bius transformations over $\C$ determines a quotient map \[\pi:\textup{Rat}_d\to\mathcal{M}_d,\] where $\mathcal{M}_d$ has the structure of an affine variety defined over $\Q$, and $\pi$ can be defined over $\Q$ \cite[\S4.4]{Silverman3}. The space $\mathcal{M}_d$ can thus be viewed as a moduli space of degree $d$ rational maps of $\mP^1(\C)$. Embedding the variety $\mathcal{M}_d$ into projective space $\mP^N$ for a suitable $N$ yields an associated height function $h_{\mathcal{M}_d}$.

For a number field $K$, let $M_K$ denote the set of places of $K$, each giving rise to a distinct absolute value, and normalized so that each $v$ restricts to the usual $v$-adic absolute value on $\Q$. Fix an embedding of $\overline{\Q}$ into $\C$, and let $v_0\in M_K$ be the corresponding valuation. Let $M_K^{\infty}$ denote the archimedean places in $M_K$, and let $M_K^0$ denote the non-archimedean places. If $\phi(z)=a_dz^d+a_{d-1}z^{d-1}+\cdots+a_1z+a_0$ and $v\in M_K^0$, then we say $\phi$ has \textit{bad reduction at} $v$ if either $|a_d|_v\ne 1$, or $|a_i|_v>1$ for some $0\le i\le d-1$ (see \cite[\S2.5]{Silverman3} for a definition that applies to rational maps). For $v\in M_K$, let $\CC_v$ denote the completion of an algebraic closure of $K_v$. For $\phi(z)\in\CC_v[z]$ and $z\in\C_v$, let \[G_{\phi,v}(z)=\lim_{n\to\infty} \fr{1}{d^n}\log\max\{1,|\phi^n(z)|_v\}\] be the standard $v$-adic escape-rate function, and let \[M_v(\phi)=\max\{G_{\phi,v}(c_i)\mid \phi'(c_i)=0\}.\] (See \S3.4, 3.5 of \cite{Silverman3} for a proof that the limit defining $G_{\phi,v}(z)$ exists.) When $v=v_0$, we will denote $M_{v_0}(\phi)$ by simply $M(\phi)$, and similarly for $G_{\phi,v_0}$. When $v$ is archimedean, the escape-rate function is harmonic wherever it is non-zero; consequently, the level sets corresponding to non-zero values are finite unions of closed curves \cite{Milnor}. Following convention, we will use the alternate notation $\hat{\lambda}_v$ instead of $G_{\phi,v}$ whenever $\hat{\lambda}_v=G_{\phi,v}$ is being referred to in its role as a local height function. Note that $G_{\phi,v}(z)$ obeys the transformation rule \begin{equation}{\label{eqn:transformation}}G_{\phi,v}(\phi(z))=dG_{\phi,v}(z)\end{equation} for all $z\in \CC_v$.

If $\phi\in K[z]$, then $\phi$ has good reduction at $v\in M_K^0$ only if $\hat{\lambda}_v(\alpha)=\log\max\{1,|\alpha|_v\}$ for all $\alpha \in K$ \cite[Proposition 1.4]{Benedetto2}. This fact will be used in the proof of Lemma \ref{lem:nonsuppd}. The \textit{canonical height} $\hat{h}_\phi(\alpha)$ of $\alpha\in K$ is defined as \[\hat{h}_\phi(\alpha)=\lim_{n\to\infty}\fr{1}{d^n}h(\phi^n(\alpha)).\] It can also be expressed as \[\hat{h}_\phi(\alpha)=\fr{1}{[K:\Q]}\sum_{v\in M_K} n_v\hat{\lambda}_v(\alpha),\] where $n_v=[K_v:\Q_v]$. The canonical height is invariant under conjugation: in other words, if $\phi=\mu^{-1}\circ \psi\circ\mu$ for a M\"{o}bius transformation $\mu\in \overline{\Q}(z)$, then $\hat{h}_\phi(\alpha)=\hat{h}_\psi(\mu^{-1}(\alpha))$.

\begin{lem}\cite[Lemma 2.1]{Ingram1}{\label{lem:Ingram1}} Let $v\in M_K$, let $d\ge2$, and let \[\phi(z)=a_dz^d+a_{d-1}z^{d-1}+\dots+a_1z+a_0\in K[z].\] If $z\in\C_v$ satisfies \[|z|_v>C_{\phi,v}:=|2d|_v\max_{0\le i\le d}\left\{1,\left|\fr{a_i}{a_d}\right|_v^{1/(d-i)},|a_d|_v^{-1/(d-1)}\right\},\] then \[G_{\phi,v}(z)=\log|z|_v+\fr{1}{d-1}\log|a_d|_v+\epsilon(\phi,v,z)\] where \[-\log2\le \epsilon(\phi,v,z)\le \log\dfrac{3}{2}\] if $v\in M_K^{\infty}$, and $\epsilon(\phi,v,z)=0$ if $v\in M_K^0$. \end{lem} In \cite{Ingram1}, Ingram relates $C_{\phi,v}$ to $M_v(\phi)$ for polynomials $\phi\in\overline{\Q}[z]$ in a particular normal form, which we now introduce. For $\textbf{c}=(c_1,\dots,c_{d-1})\in\mathbb{A}^{d-1}(\overline{\Q})$ and $d\ge 2$, set \[f_{\textbf{c}}(z)=\fr{1}{d}z^d-\fr{1}{d-1}(c_1+\dots+c_{d-1})z^{d-1}+\dots+(-1)^{d-1}c_1c_2\cdots c_{d-1}z,\] so that \[f_{\textbf{c}}'(z)=\prod_{i=1}^{d-1}(z-c_i).\] 

\begin{lem}\cite{Ingram1}{\label{lem:Ingram2}} Let $C_{f_\textbf{c},v}$ be as in Lemma \ref{lem:Ingram1}. There is a constant $\xi_v$ depending only on $d$ and on $v$ such that \[\log C_{f_{\textbf{c}},v}\le M_v(f_{\textbf{c}})+\xi_v\] for all $f_{\textbf{c}}$. Moreover, $\xi_v=0$ for all but finitely many $v\in M_K$.
	
\end{lem}

\begin{proof}
	The proof follows immediately from Lemmas 2.2 and 2.5 of \cite{Ingram1}.
\end{proof}

Finally, we show that any $\phi\in K[z]$ has a conjugate of the form $f_\textbf{c}$, with field of definition having degree bounded in terms of $d$ and $[K:\Q]$.

\begin{lem}{\label{lem:normalformconjugation}} Every $\phi\in\overline{\Q}[z]$ of degree $d\ge2$ is affine conjugate to a polynomial of the form $f_\textbf{c}$ for some $\textbf{c}\in\mathbb{A}^{d-1}(\overline{\Q})$. If $\phi$ is defined over a number field $K$, then $f_\textbf{c}$ is defined over a number field $L$ such that $[L:K]\le d(d-1)$. 
\end{lem}

\begin{proof}
	Let $\phi(z)=a_dz^d+a_{d-1}z^{d-1}+\cdots+a_1z+a_0$, and let $\gamma$ be a fixed point of $\phi$. If $\mu=(a_dd)^{-1/(d-1)}z+\gamma$, then $g=\mu^{-1}\circ \phi\circ\mu$ is of the form $f_\textbf{c}$. Indeed, it is easy to check that the derivative $g'$ is monic, and that $g$ has a fixed point at $0$. But the form of $f_\textbf{c}$ is characterized by exactly these two properties. Finally, $f_\textbf{c}$ is defined over $L=K(\gamma,(a_dd)^{-1/(d-1)})$, and $[L:K]\le d(d-1)$.
\end{proof}

\section{Local height bounds: Archimedean case}

We remind the reader of standard facts from complex analysis, which will be invoked in the proof of Proposition \ref{prop:annuli}. By the Poincar\'{e}-Koebe uniformization theorem, any doubly connected domain in $\CC$ is biholomorphic to an open round annulus. We will thus refer to any such domain as an \textit{annulus}. If an annulus $A_1$ is biholomorphic to the bounded round annulus $A_2=\{z\in\C\mid r_1<|z|<r_2\}$, we define the $\textit{modulus}$ of $A_1$ (and $A_2$) to be \[\textup{mod}(A_1)=\fr{1}{2\pi}\log\left(\fr{r_2}{r_1}\right),\] where the modulus is infinite if $r_1=0$. Finally, we say an annulus $A_1$ is \textit{essentially embedded} in an annulus $A_2$ if $A_1\subset A_2$, and the bounded complementary component of $A_2$ is contained in the bounded complementary component of $A_1$.

\begin{lem}\cite[\S 9.3.5]{Remmert}{\label{lem:kto1annulus}} If $f:A_1\to A_2$ is a degree $k$ holomorphic covering map of annuli in $\C$, then \[\textup{mod}(A_1)=\dfrac{1}{k}\textup{mod}(A_2).\]
\end{lem}

\begin{lem}[\cite{Milnor}, Corollary B.6]{\label{lem:grotzsch}} If $A_1,A_2,\dots,A_m$ are disjoint essentially embedded subannuli of the annulus $A\subset\C$, then \[\textup{mod}(A_1)+\dots+\textup{mod}(A_m)\le\textup{mod}(A).\] \end{lem} We now use these properties of annuli and holomorphic covering maps between them to prove a key proposition.

\begin{prop}{\label{prop:annuli}} Let $\phi\in\C[z]$ be of degree $d\ge2$, and assume $M=M(\phi)>0$. Let $\mathcal{C}_2$ be a connected component of the level set $G_\phi(z)=M/d^2$. Then the annulus bounded by $\mathcal{C}_2$ and the curve $G_\phi(z)=dM$ has modulus at least $\fr{1}{2\pi}(d+\fr{1}{d-1})M$.\end{prop}

\begin{proof} On the set $\{z\in\C\mid G_\phi(z)>M\}$, the function $G_\phi(z)$ equals $\log|\tau_\phi|$, where $\tau_\phi$ is the B\"{o}ttcher coordinate near $\infty$. Therefore $\tau_\phi$ takes the region $M<G_\phi(z)<dM$ to the region $M<\log|z|<dM$. As $\tau_\phi$ is an isomorphism, the annulus $A_0$ bounded by the curves $G_\phi(z)=dM$ and $G_\phi(z)=M$ has modulus $\fr{1}{2\pi}(d-1)M$. Removing from $A_0$ the level curves of any elements in the grand orbit of a critical point of $\phi$, we obtain disjoint \textit{fundamental subannuli} $A_{0,1},\dots,A_{0,j}$ of $A_0$. Again using the B\"{o}ttcher coordinate, we observe that \begin{equation}{\label{eqn:equalityA1}}\sum_{i=1}^j\textup{mod}(A_{0,i})=\textup{mod}(A_0).\end{equation} Now let $\mathcal{C}_2$ be any connected component of the level curve $G_\phi(z)=M/d^2$, and let $\mathcal{C}_1$ be the unique connected component of the level curve $G_\phi(z)=M/d$ which bounds a region in $\C$ containing $\mathcal{C}_2$. Let $A_2$ be the annulus with boundaries $\mathcal{C}_1$ and $\mathcal{C}_2$. Within $A_2$, there is a collection $A_{2,1},\dots,A_{2,j}$ of essentially embedded subannuli, where for each $1\le i\le j$, $A_{2,i}$ is mapped by $\phi^2$ onto the fundamental subannuli $A_{0,1},\dots,A_{0,j}$ respectively. By Lemma \ref{lem:kto1annulus}, we have \begin{equation}{\label{eqn:modA3}}\textup{mod}(A_{2,i})\ge \fr{1}{(d-1)^2}\textup{mod}(A_{0,i})\end{equation} for all $1\le i\le j$. Combining (\ref{eqn:equalityA1}) with (\ref{eqn:modA3}) and applying Lemma \ref{lem:grotzsch}, we get \[\textup{mod}(A_2)\ge \fr{1}{2\pi}\fr{M}{d-1}.\] A similar argument shows that $\textup{mod}(\phi(A_2))\ge \fr{1}{2\pi}M$. We conclude that the modulus of the annulus bounded by $G_\phi(z)=dM$ and $\mathcal{C}_2$ is at least $\fr{1}{2\pi}\left((d-1)+1+\fr{1}{d-1}\right)M=\fr{1}{2\pi}\left(d+\fr{1}{d-1}\right)M$.

\end{proof}

Proposition \ref{prop:annuli} implies a corollary about pre-images of low height points under $f_\textbf{c}^3$. We first make the following definition.

\begin{definition*} For $C_{f_\textbf{c},v}$ as in Lemma \ref{lem:Ingram1}, let \[B_{f_\textbf{c},v}(\infty)=\CC_v\backslash\overline{D}(0,C_{f_\textbf{c},v}),\] where $\overline{D}(0,C_{f_\textbf{c},v})=\{z\in\CC_v:|z|_v\le C_{f_\textbf{c},v}\}.$\end{definition*}

We use the notation $C_{f_\textbf{c}}$ and $B_{f_\textbf{c}}(\infty)$ when considering $v=v_0$, the restriction to $\overline{\Q}$ of the standard complex absolute value. (Recall that we have fixed an embedding of $\overline{\Q}$ into $\C$.)

\begin{cor}{\label{cor:preimagesmash}}
	Let $\alpha\in \C$, and let $d\ge2$. There exists a constant $\delta=\delta(d)$ depending only on $d$, so that for every $\textbf{c}\in\mathbb{A}^{d-1}(\overline{\Q})$, $\alpha\notin B_{f_\textbf{c}}(\infty)$ implies that every $y\in f_\textbf{c}^{-3}(\alpha)$ satisfies \[\min_{\beta\in f_\textbf{c}^{-3}(0)}\log|y-\beta|\le -\fr{1}{d-1}M(f_\textbf{c})+\delta.\] 
\end{cor} \begin{proof} By Lemmas \ref{lem:Ingram1} and \ref{lem:Ingram2}, we see that for all $f_{\textbf{c}}$ with $M(f_\textbf{c})$ sufficiently large, \[\{z\in\CC\mid G_{f_\textbf{c}}(z)>dM(f_\textbf{c})\}\subset B_{f_\textbf{c}}(\infty).\] Lemma \ref{lem:Ingram1} then implies that for all $f_\textbf{c}$ with $M=M(f_\textbf{c})$ sufficiently large, $\{z\in\CC\mid G_{f_\textbf{c}}(z)\le dM(f_\textbf{c})\}$ is contained in a disk of radius comparable to $e^{dM(f_\textbf{c})}$. If $\alpha\notin B_{f_\textbf{c}}(\infty)$, then by (\ref{eqn:transformation}), any $y\in f_\textbf{c}^{-3}(\alpha)$ is contained in the closed region $D_1$ bounded by some connected component $\mathcal{C}_y$ of the curve $G_{f_\textbf{c}}(z)=M/d^2$. Let $D_2=\{z\in\C\mid G_{f_\textbf{c}}(z)\le dM\}$, and let $A(f_\textbf{c}^3)$ be the annulus in $\CC$ bounded by $G_{f_\textbf{c}}(z)=dM$ and $\mathcal{C}_y$. Every annulus $A\subset\CC$ of sufficiently large modulus contains an essentially embedded round annulus of modulus differing by at most a constant from the modulus of $A$ \cite{McMullen}. Applying Proposition \ref{prop:annuli}, it follows that for $M(f_\textbf{c})$ sufficiently large, there is an essentially embedded round annulus $A'(f_\textbf{c}^3)$ contained in $A(f_\textbf{c}^3)$ with outer boundary of radius comparable to $e^{dM(f_\textbf{c})}$, and modulus at least $\fr{1}{2\pi}(d+\fr{1}{d-1})M(f_\textbf{c})-O(1)$. We now observe that $f_\textbf{c}^3$ maps $D_1$ onto $D_2$, and that $0\in D_2$. From this, we conclude that there exist an $m$ and a $C$ such that for all degree $d$ polynomials $f_{\textbf{c}}\in\overline{\Q}[z]$ with $M(f_{\textbf{c}})\ge m$, and for all $\alpha\in\overline{\Q}$, $\alpha\notin B_{f_\textbf{c}}(\infty)$ implies that each $y\in f_\textbf{c}^{-3}(\alpha)$ satisfies \[|y-\beta|\le \dfrac{C}{\textup{exp}(\fr{1}{d-1}M(f_{\textbf{c}}))}\] for some root $\beta\in\overline{\Q}$ of $f_\textbf{c}^3$. On the other hand, if $M(f_\textbf{c})<m$, then by Lemma \ref{lem:Ingram2}, \[\log C_{f_\textbf{c}}\le m+\xi_{v_0},\] so by the triangle inequality, \[\log|y-\beta|<\log2+m+\xi_{v_0}\le \log2+\delta'.\] As $m$ and $\xi_{v_0}$ depend only on $d$, it follows that $\delta'$ can be chosen to depend only on $d$. \end{proof}

\section{Local height bounds: non-archimedean case}

For a number field $K$ and $v\in M_K^0$, recall that $\C_v$ denotes the completion of an algebraic closure $\overline{K_v}$ of $K_v$. A \textit{disk} in $\C_v$ is a set of the form $D=\{z\in\C_v\mid |z-a|_v<r\}$ for some $a\in\C_v$, where $r>0$. For such a disk $D$, we write $\overline{D}$ to denote the set $\{z\in\C_v\mid|z-a|_v\le r\}$. An \textit{annulus} in $\C_v$ is a set of the form $\{z\in\C_v\mid r_1<|z-a|_v<r_2\}$ for some $a\in\C_v$, and some non-negative $r_1<r_2$. We define the modulus of such an annulus to be $\fr{1}{2\pi}\log(r_2/r_1)$, by analogy with the archimedean case. We have the following proposition, which shows that moduli of annuli transform the same way under covering maps as in the archimedean setting.

\begin{prop}\cite[Corollary 2.6]{Baker}{\label{prop:nonarchmoduli}} If $v\in M_K^0$, and $\phi\in\C_v(z)$ is a degree $k$ covering map $\phi:A_1\to A_2$ of annuli in $\C_v$, then \[\textup{mod}(A_1)=\fr{1}{k}\textup{mod}A_2.\] \end{prop} We will also make use of the fact that pre-images of disks under polynomials behave nicely in the non-archimedean setting. \begin{prop}\cite[Lemma 2.7]{Benedetto}{\label{prop:nonarchdiskpreimage}} Let $D\subset\C_v$ be a disk, and let $\phi\in\C_v[z]$ be a polynomial of degree $d\ge 1$. Then $\phi^{-1}(D)$ is a disjoint union $D_1\cup\cdots\cup D_m$ of disks, with $1\le m\le d$. Moreover, for each $i=1,\dots,m$, there is an integer $1\le d_i\le d$ such that every point in $D$ has exactly $d_i$ pre-images in $D$, and that $d_1+\cdots+d_m=d$.
\end{prop} Propositions \ref{prop:nonarchmoduli} and \ref{prop:nonarchdiskpreimage} imply a non-archimedean counterpart of Corollary \ref{cor:preimagesmash}.
\begin{prop}{\label{prop:badredpreimagesquash}} Let $v\in M_K^0$, let $\alpha \in \CC_v$, and let $d\ge 2$. There exists a constant $\delta_v$ depending only on $d$ and on $v$ such that for every $f_\textbf{c}(z)\in \overline{\Q}[z]$ of degree $d$, $\alpha\notin B_{f_\textbf{c},v}(\infty)$ implies that every $y\in f_\textbf{c}^{-3}(\alpha)$ satisfies \[\min_{\beta\in f_\textbf{c}^{-3}(0)}\log|y-\beta|_v\le -\fr{1}{d-1}M_v(f_\textbf{c})+\delta_v.\]
\end{prop}
\begin{proof} Lemma \ref{lem:Ingram1} implies that there exists an $m_v$ such that for all $f_\textbf{c}$ with $M=M_v(f_\textbf{c})\ge m_v$, the level set $G_{f_\textbf{c},v}(z)=dM$ is in $B_{f_\textbf{c},v}(\infty)$, and the distortion factor $\epsilon(f_\textbf{c},v,z)$ is 0. As Lemma \ref{lem:Ingram2} specifies that there are finitely many non-zero $\xi_v$, this $m_v$ can be chosen to be independent of $v\in M_K^0$. We denote it by $m_0$.
	
From Lemma \ref{lem:Ingram1}, it follows that when $M_v(f_\textbf{c})\ge m_0$, the level set $G_{f_\textbf{c},v}(z)=dM$ is a set of the form $|z|_v=R$. For any connected component (disk) pre-image $D_1$ of $D_2:=\{z\in\CC_v\mid G_{f_\textbf{c},v}(z)<dM\}$ under $f_\textbf{c}^3$, $D_2-\overline{D_1}$ is an annulus $A(f_\textbf{c}^3)$. A similar proof as the proof of Proposition \ref{prop:annuli} shows that the modulus of $A(f_\textbf{c}^3)$ is at least $\fr{1}{2\pi}(d+\fr{1}{d-1})M_v(f_\textbf{c})$ (note that this annulus is round, as non-archimedean disks are centerless, so we do not need to invoke anything akin to the Gr\"{o}tzsch inequality). From Proposition \ref{prop:nonarchdiskpreimage}, we know that $f_\textbf{c}^3$ maps $D_1$ onto $D_2$. Thus, if $\alpha\notin B_{f_\textbf{c},v}(\infty)$ where $M_v(f_\textbf{c})\ge m_0$, then for every $y\in f_\textbf{c}^{-3}(\alpha)$, we have \begin{equation}{\label{eqn:largemv}}\min_{\beta\in f_\textbf{c}^{-3}(0)}|y-\beta|_v\le \textup{exp}\left(-\fr{1}{d-1}M_v(f_\textbf{c})\right).\end{equation}
	
On the other hand, when $M_v(f_\textbf{c})<m_0$, we have \[\log C_{f_\textbf{c},v}<m_0+\xi_v,\] so by the ultrametric inequality, if $\alpha\notin B_{f_\textbf{c},v}(\infty)$, then every $y\in f_\textbf{c}^{-3}(\alpha)$ satisfies \begin{equation}{\label{eqn:smallmv}}\log|y-\beta|_v<m_0+\xi_v\le \delta_v\end{equation} for some $\delta_v$ depending only on $d$ and on $v$. Combining (\ref{eqn:largemv}) and (\ref{eqn:smallmv}) completes the proof.

\end{proof}

\section{Key Inequalities}

In this section, as well as in \S\ref{section:mainthm}, the main ideas are inspired by those used by Ingram in \cite{Ingram2}. Let $d\ge 2$. For $\phi(z)=a_dz^d+a_{d-1}z^{d-1}+\dots+a_1z+a_0\in K[z]$ and $v\in M_K$, let \[\lambda_v(\phi)=\max\{\lambda_v(a_i)\mid 0\le i\le d\},\] where $\lambda_v(a_i)=\log\max\{1,|a_i|_v\}$, let $n_v=[K_v:\Q_v]$, and let \[h(\phi)=\dfrac{1}{[K:\Q]}\sum_{v\in M_K} n_v\lambda_v(\phi).\] Let $\phi\in K[z]$ be conjugate to $f_\textbf{c}\in L[z]$, where $L$ is as in Lemma \ref{lem:normalformconjugation}. From this point onward, we denote $f_\textbf{c}$ by $f$ for ease of notation. 

\begin{lem}[\cite{Ingram1}]{\label{lem:httoesc}} For any $v\in M_L$, we have \[\lambda_v(f)\le dM_v(f)+\eta\] for some constant $\eta$ depending only on $d$. \end{lem}
\begin{proof}
	This results from combining Lemmas 2.1, 2.2 and 2.5 of \cite{Ingram1}. 
\end{proof}

\begin{lem}{\label{lem:alphabasin}}
	Let $v\in M_L$. Let $\alpha\in B_{f,v}(\infty)\cap L$. Then there exists a constant $C_1$ depending only on $d$ such that if $j>i$, then \[\lambda_v(f)+d(d-1)\log|f^i(\alpha)-f^j(\alpha)|_v\le d^{j+3}\hat{\lambda}_v(\alpha)+C_1.\]
\end{lem}

\begin{proof}
	From the definition of $C_{f,v}$, we have \[\log C_{f,v}\ge \log(2d)+\fr{1}{d}\lambda_v(f).\] Thus, \[\alpha\in B_{f,v}(\infty)\llra \lambda_v(\alpha)>\lambda_v(C_{f,v})\ge \log(2d)+\fr{1}{d}\lambda_v(f).\] If $\alpha\in B_{f,v}(\infty)\cap L$, then, this yields \begin{equation}{\label{eqn:ht1}}\hat{\lambda}_v(\alpha)\ge \log(2d)+\fr{1}{d}\lambda_v(f)-\log2\end{equation} by Lemma \ref{lem:Ingram1}. Moreover, by Lemma \ref{lem:Ingram1} and the triangle inequality, \begin{equation}\begin{split}{\label{eqn:ht2}}d(d-1)\log|f^i(\alpha)-f^j(\alpha)|_v & \le d(d-1)(\log2+\hat{\lambda}_v(f^j(\alpha))+\log2) \\ & =d(d-1)(d^j\hat{\lambda}_v(\alpha)+2\log2).\end{split}\end{equation} Therefore, rearranging (\ref{eqn:ht1}) and combining with (\ref{eqn:ht2}) gives \begin{equation*}\begin{split}\lambda_v(f)+d(d-1)\log|f^i(\alpha)-f^j(\alpha)|_v & \le d\hat{\lambda}_v(\alpha)+d\log2-d\log(2d) \\ & +d(d-1)(d^j\hat{\lambda}_v(\alpha)+2\log2) \\ & \le (d^{j+1}(d-1)+d)\hat{\lambda}_v(\alpha)+C_1 \\ & \le d^{j+3}\hat{\lambda}_v(\alpha)+C_1\end{split}\end{equation*} whenever $\alpha\in B_{f,v}(\infty)$.
\end{proof}

\begin{prop}{\label{prop:pigeonhole}}
	Let $\alpha\in L$ have infinite orbit under $f$, let $X$ be a finite set of positive integers, and let $v\in M_L$. Then there exists $Y\subset X$ containing at least $\dfrac{\#X-3}{d^3+1}$ values such that for all $i,j\in Y$ with $j>i$, we have \[\lambda_v(f)+d(d-1)\log|f^i(\alpha)-f^j(\alpha)|_v\le d^{j+3}\hat{\lambda}_v(\alpha)+C_2\] for some constant $C_2$ depending only on $d$.
	
\end{prop}

\begin{proof}
	Suppose that at least $\fr{\#X-3}{d^3+1}$ values $k\in X$ have $f^k(\alpha)\in B_{f,v}(\infty)$, and let $Y$ be the set of such values. If $i,j\in Y$ with $j>i$, then Lemma \ref{lem:alphabasin} applied to $f^i(\alpha)$ implies \begin{equation*}\begin{split} \lambda_v(f)+d(d-1)\log|f^i(\alpha)-f^{j-i}(f^i(\alpha))| & \le d^{j-i+3}\hat{\lambda}_v(f^i(\alpha))+C_1 \\ &= d^{j+3}\hat{\lambda}_v(\alpha)+C_1.\end{split}\end{equation*}
	
	Now suppose that fewer than $\fr{\#X-3}{d^3+1}$ values $k\in X$ satisfy $f^k(\alpha)\in B_{f,v}(\infty)$. Then there are more than $\fr{d^3(\#X-3)}{d^3+1}+3$ values $k\in X$ such that $f^k(\alpha)\notin B_{f,v}(\infty)$, so more than $\fr{d^3}{d^3+1}(\#X-3)$ such that $f^{k+3}(\alpha)\notin B_{f,v}(\infty)$. By Corollary \ref{cor:preimagesmash}, Proposition \ref{prop:badredpreimagesquash}, and the pigeon-hole principle, there is a $\beta\in \overline{\Q}$ with $f^3(\beta)=\beta$ and
	
	\[\log|f^k(\alpha)-\beta|_v\le \fr{-1}{d-1}M_v(f)+\max\{\delta,\delta_v\}\] for at least $\fr{\#X-3}{d^3+1}$ values $k\in X$. Applying Lemma \ref{lem:httoesc} to the right-hand side, this gives \begin{equation*}\begin{split}\log|f^k(\alpha)-\beta|_v & \le \fr{1}{d(d-1)}(\eta-\lambda_v(f))+\max\{\delta,\delta_v\}\\ 
	& =\fr{-\lambda_v(f)}{d(d-1)}+\fr{\eta}{d(d-1)}+\max\{\delta,\delta_v\}\end{split}\end{equation*}
 for at least $\fr{\#X-3}{d^3+1}$ values $k\in X$. If $i$ and $j$ are two of these values, then by the triangle inequality, \[\log|f^i(\alpha)-f^j(\alpha)|_v\le \fr{-\lambda_v(f)}{d(d-1)}+\fr{\eta}{d(d-1)}+\max\{\delta,\delta_v\}+\log2\] where $\delta,\eta,$ and $\delta_v$ depend only on $d$. Hence \begin{equation*}\lambda_v(f)+d(d-1)\log|f^i(\alpha)-f^j(\alpha)|_v\le \eta+d(d-1)(\max\{\delta,\delta_v\}+\log2).\end{equation*} Setting \[C_2=C_1+\eta+d(d-1)(\max\{\delta,\delta_v\}+\log2)\] completes the proof. \end{proof}

\begin{lem}{\label{lem:nonsuppd}} Let $v$ be a non-archimedean place of $L$ such that $f\in L[z]$ has good reduction at $v$. Then for all $\alpha\in L$ with $\hat{h}_f(\alpha)\ne 0$, and all $j>i$, we have \[\lambda_v(f)+d(d-1)\log|f^i(\alpha)-f^j(\alpha)|_v\le (d-1)d^{j+1}\hat{\lambda}_v(\alpha).\] \end{lem}

\begin{proof}
	Since $f$ has good reduction at $v$, $\lambda_v(f)=0$, so it suffices to show that \[d(d-1)\log|f^i(\alpha)-f^j(\alpha)|_v\le (d-1)d^{j+1}\hat{\lambda}_v(\alpha).\] But this is clear: $f$ having good reduction at $v$ implies that $\hat{\lambda}_v(\alpha)=\lambda_v(\alpha)$ for all $\alpha\in L$, and so \[\log|f^i(\alpha)-f^j(\alpha)|_v\le\log\max\{|f^i(\alpha)|_v,|f^j(\alpha)|_v,1\}= \max\{\hat{\lambda}_v(f^i(\alpha)),\hat{\lambda}_v(f^j(\alpha))\}=d^j\hat{\lambda}_v(\alpha)\] for all $j>i$.
\end{proof}

\section{Proof of Main Theorem}{\label{section:mainthm}}

\begin{proof}[Proof of Theorem \ref{thm:minht}] Let $\phi\in K[z]$ be affine conjugate to $f=f_\textbf{c}\in L[z]$, where $L$ is as in Lemma \ref{lem:normalformconjugation}, and let $\alpha\in L$ be such that $\hat{h}_{f}(\alpha)\ne 0$. Let \[N=\fr{3}{d^3}((d^3+3)^{r+s+1}-1),\] where $r$ is the number of distinct archimedean absolute values of $L$, and $s$ is the number of $v\in M_L^0$ of bad reduction. Denote the union of the set of places of bad reduction of $f_\textbf{c}$ and the archimedean places of $f$ by $M_L^1$. Let $X=\{1,2,\dots,N\}$, and let $v\in M_L^1$. It is easy to show that \[\fr{\#X-3}{d^3+1}\ge \fr{3}{d^3}((d^3+3)^{r+s}-1).\] Hence, by Proposition \ref{prop:pigeonhole}, we can choose a subset $X'\subset X$ such that \[\#X'\ge \fr{\#X-3}{d^3+1}\ge \fr{3}{d^3}((d^3+3)^{r+s}-1)\] and such that for all $i,j\in X'$ with $j>i$, we have \begin{equation}{\label{eqn:PH}}\lambda_v(f)+d(d-1)\log|f^i(\alpha)-f^j(\alpha)|_v\le d^{j+3}\hat{\lambda}_v(\alpha)+C_2\end{equation} for some constant $C_2$ depending only on $d$. By induction on the size of the resulting subsets, we obtain a set $Y\subset X$ with \[\#Y\ge\fr{3}{d^3}((d^3+3)^1-1)> 3\] such that (\ref{eqn:PH}) holds for every $v\in M_L^1$ and for all $j>i\in Y$. From this, we obtain \begin{equation}{\label{eqn:arch}}\sum_{v\in M_{L}^1}n_v[\lambda_v(f)+d(d-1)\log|f^i(\alpha)-f^j(\alpha)|_v] \le \sum_{v\in M_{L}^1}n_v[d^{j+3}\hat{\lambda}_v(\alpha)+C_2]\end{equation} for $j>i\in Y$, and $n_v=[L_v:\Q_v]$. From Lemma \ref{lem:nonsuppd}, we also have \begin{equation}{\label{eqn:nonarch}}\sum_{v\in M_L^0-M_L^1}n_v[\lambda_v(f)+d(d-1)\log|f^i(\alpha)-f^j(\alpha)|_v]\le n_vd^{j+3}\sum_{v\in M_{L}^0-M_L^1}\hat{\lambda}_v(\alpha).\end{equation} Summing (\ref{eqn:arch}) and (\ref{eqn:nonarch}) and applying the product formula, we deduce that \begin{equation*}\begin{split}[L:\Q]h(f) & = \sum_{v\in M_L} n_v(\lambda_v(f)+d(d-1)\log|f^i(\alpha)-f^j(\alpha)|_v)\\ & \le\sum_{v\in M_L}n_v(d^{j+3}\hat{\lambda}_v(\alpha))+\sum_{v\in M_L^1}n_vC_2 \\ & \le [L:\Q]d^{j+3}\hat{h}_f(\alpha)+[L:\Q](r+s)C_2\\ & \le [L:\Q]d^{N+3}\hat{h}_f(\alpha)+[L:\Q](r+s)C_2\end{split}\end{equation*} where $C_2$ depends only on $d$. In particular, \begin{equation}{\label{eqn:normalformresult}}h(f)<<\min\{\hat{h}_f(\alpha)\mid \alpha\in L, \hat{h}_f(\alpha)\ne 0\}.\end{equation} Now let \[h_{min}(\langle f\rangle)=\min\{h(g)\mid g\sim f, g\in\overline{\Q}[z]\},\] where $g\sim f$ if and only if $g$ is conjugate to $f$ by a M\"{o}bius transformation over $\overline{\Q}$. By the proof of \cite[Lemma 6.32]{Silverman}, we have \[h_{min}(\langle f\rangle)\asymp h_{\mathcal{M}_d}(\langle f\rangle).\] Thus (\ref{eqn:normalformresult}) implies \begin{equation}\begin{split}{\label{eqn:moduliht}} h_{\mathcal{M}_d}(\langle \phi\rangle)=h_{\mathcal{M}_d}(\langle f\rangle) & <<\min\{\hat{h}_f(\alpha)\mid \alpha\in L, \hat{h}_f(\alpha)\ne 0\} \\ & \le \min\{\hat{h}_\phi(\alpha)\mid \alpha\in K, \hat{h}_\phi(\alpha)\ne 0\},\end{split}\end{equation} where the implied constants depend only on $d$, $[L:\Q]$, and the number of places of bad reduction of $f=f_\textbf{c}$. (Recall that we fixed a choice of $h_{\mathcal{M}_d}$ in advance.) Finally, we note that the number of places of bad reduction of $f_{\textbf{c}}$ is at most the number of places of bad reduction of $\phi$ plus the number of primes dividing $d$. This can be seen from the conjugation taking $\phi$ to $f_\textbf{c}$, and from the fact that $\phi(z)=a_dz^d+a_{d-1}z^{d-1}+\cdots+a_1z+a_0$ has bad reduction at all primes dividing $a_d$. From Lemma \ref{lem:normalformconjugation} we had $[L:K]\le d(d-1)$, so the implied constants depend only on $d$, $[K:\Q]$, and the number of places of bad reduction of $\phi$. This completes the proof. \end{proof}


\begin{thebibliography}{1}
	
	\bibitem{Baker} M. Baker, S. Payne and J. Rabinoff. On the structure of non-Archimedean analytic curves. \emph{Contemporary Mathematics} 605 (2013) 93-121.
	
	\bibitem{Benedetto2} R. Benedetto, B. Dickman, S. Joseph, B. Krause, D. Rubin, and X. Zhou. Computing points of small height for cubic polynomials. \emph{Involve} 2 (2009) 37-64.
	
	\bibitem{Benedetto} R. Benedetto. Preperiodic points of polynomials over global fields. \emph{Journal f\"{u}r die Reine und Angewandte Mathematik} 608 (2007) 123-153. 
	
	\bibitem{BrannerHubbard} B. Branner and J. Hubbard. The iteration of cubic polynomials. {II}. {P}atterns and
	parapatterns. \emph{Acta Mathematica} 169 (1992) 229-325.
	
	\bibitem{CallSilverman} G. Call and J. Silverman. Canonical heights on varieties with morphisms. \emph{Compositio Mathematica} 89.2 (1993) 163-205.
	
	\bibitem{DeMarco} L. DeMarco. Finiteness for degenerate polynomials. Holomorphic Dynamics and Renormalization: A Volume in Honour of John Milnor's 75th birthday. Fields Institute Communications, AMS, 53 (2008) 89-104.
	
	\bibitem{DM} L. DeMarco and C. McMullen. Trees and the dynamics of polynomials. \emph{Annales Scientifiques de l'\'{E}cole Normale Sup\'{e}rieure} 41 (2008) 337-383.
	
	\bibitem{HindrySilverman} M. Hindry and J. Silverman. The canonical height and integral points on elliptic curves. \emph{Inventiones Mathematicae} 93 (1998) 419-450.
	
	\bibitem{Ingram1} P. Ingram. A finiteness result for post-critically finite polynomials. \emph{International Mathematics Research Notices} 3 (2012) 524-543.
	
	\bibitem{Ingram2} P. Ingram. Lower bounds on the canonical height associated to the morphism $z^d+c$. \emph{Monatschefte f\"{u}r Mathematik} 157 (2007) 69-89.
	
	\bibitem{Lang} S. Lang. \emph{Elliptic curves: Diophantine Analysis}, volume 231 of Grundlehren der Mathematischen Wissenschaften, Springer-Verlag, Berlin, 1978.
	
	\bibitem{McMullen} C. McMullen. \emph{Complex Dynamics and Renormalization}, volume 135 of Annals of Mathematical Studies, Princeton University Press, 1994.
	
	\bibitem{Milnor} J. Milnor. \emph{Dynamics in One Complex Variable}, volume 10 of Annals of Mathematical Studes, Princeton University Press, 2006.
	
	\bibitem{Remmert} R. Remmert. \emph{Classical Topics in Complex Function Theory}, volume 172 of Graduate Texts in Mathematics, Springer, 1998.
	
	\bibitem{Silverman3} J. Silverman. \emph{The Arithmetic of Dynamical Systems}, volume 241 of Graduate Texts in Mathematics, Springer, 2007.
	
	\bibitem{Silverman2} J. Silverman. Lower bound for the canonical height on elliptic curves. \emph{Duke Mathematical Journal} 48 (1981) 633-648.
	
	\bibitem{Silverman}  J. Silverman. \emph{Moduli Spaces and Arithmetic Dynamics}, volume 30 of CRM Monograph Series, AMS, 2012.
	
\end{thebibliography}
\end{document}